\documentclass[a4paper,11pt]{article}
\usepackage[utf8]{inputenc}
\usepackage[UKenglish]{babel}
\usepackage{enumitem}
\usepackage{graphicx}
\usepackage{tikz}
\usepackage{geometry}
\usepackage{amsmath}
\usepackage{amssymb}
\usepackage{amsthm}
\usepackage{thmtools}
\usepackage{hyperref}

\usetikzlibrary{arrows}
\usetikzlibrary{decorations.markings,decorations.pathreplacing}

\title{Abstract commensurability and the Gupta--Sidki group}
\author{Alejandra Garrido\\
}
\date{}

\DeclareMathOperator{\St}{St}
\DeclareMathOperator{\Aut}{Aut}

\newcommand{\lf}{\leq_{\mathrm{f}}}
\newcommand{\ls}{\leq_{\mathrm{s}}}

\declaretheorem[name=Theorem]{thm}
\declaretheorem[name=Lemma,numberwithin=section]{lem}
\declaretheorem[name=Proposition, sibling=lem]{prop}

\begin{document}

\maketitle

\begin{abstract}

We study the subgroup structure of the infinite torsion $p$-groups defined by Gupta and Sidki in 1983. 
In particular, following results of Grigorchuk and Wilson for the first Grigorchuk group, we show that all infinite finitely generated subgroups of the Gupta--Sidki 3-group $G$ are abstractly commensurable with $G$ or $G\times G$.
 As a consequence, we show that $G$ is subgroup separable and from this it follows that its membership problem is solvable. 

Along the way, we obtain a characterization of finite subgroups of $G$ and establish an analogue for the Grigorchuk group. 

\end{abstract}

\section{Introduction}

Groups of automorphisms of regular rooted trees have received considerable attention in the last few decades. 
Their interest derives from the striking properties of some of the first examples studied, the Grigorchuk group \cite{grig} and the Gupta--Sidki $p$-groups \cite{guptasidki}. 
These groups are easily understood examples of infinite finitely generated torsion groups, answering the General Burnside Problem. 
Furthermore, the Grigorchuk group was the first group to be shown to be of intermediate word growth  and to be amenable but not elementary amenable (see \cite{GrigIntermediate}). 
Amenability of the Gupta--Sidki $p$-groups (among many other examples) was proved in \cite{AmenableBoundedAutomata}.
These and other striking results prompt one to ask what other unusual properties these groups may have.

In \cite{griwil}, the authors establish another notable result about the Grigorchuk group, 
namely that all its infinite finitely generated subgroups are (abstractly) commensurable with the group itself. 
Recall that two groups are \emph{(abstractly) commensurable} if they have isomorphic subgroups of finite index. 
This notion translates into geometric terms as ``having a common finite-sheeted covering": two spaces which have a common finite-sheeted covering have commensurable fundamental groups.
For this reason it is an important concept in geometric group theory. 
It also appears in the study of lattices in semisimple Lie groups, in profinite groups and other areas of group theory. 
Having only one commensurability class of infinite finitely generated subgroups is a very strong restriction on subgroup structure and examples where it is known to hold are scarce. 
Following a similar general strategy to that in \cite{griwil}, we will show that an analogous result holds for the Gupta--Sidki 3-group:

\begin{thm}\label{mainthm}
Every infinite finitely generated subgroup of the Gupta--Sidki {\rm 3}-group $G$ is commensurable with $G$ or $G\times G$. 
\end{thm}

This raises two questions.
The first is whether the commensurability classes are actually distinct; by contrast, the Grigorchuk group is commensurable with its direct square.
This is the motivation for the work carried out in \cite{MeWil}. 
In that paper, more general results on the structure of subgroups of branch groups yield that the classes are indeed distinct, for many examples of groups acting on $p$-regular trees where $p$ is an odd prime. 

The second question concerns the restriction to $p=3$ in \autoref{mainthm}.
It seems likely that this restriction is unnecessary.
However, our proof relies heavily on a delicate length reduction argument that is only available for $p=3$. 

Our first main theorem allows us to prove another result about the subgroup structure of $G$, also parallel to a result in \cite{griwil}.
\begin{thm}\label{thm2}
The Gupta--Sidki {\rm 3}-group is subgroup separable and hence has solvable membership problem. 
\end{thm}

Recall that a group is \emph{subgroup separable} (or LERF) if each of its finitely generated subgroups is an intersection of subgroups of finite index. 
This condition is strong and is only known to hold in very special cases such as free groups, surface groups, polycyclic groups and some 3-manifold groups (see \cite{HallFreeLERF, ScottSurfaceLERF, MalcevPolycyclicLERF, 3manifoldLERF}).
It is related to the membership problem (or generalized word problem).
The membership problem for a finitely generated group $H$ is solvable if there is an algorithm which 
given a finitely generated subgroup $K\leq H$ and an element $h\in H$ determines whether or not $h\in K$.
If a group is subgroup separable and all its finite quotients can be effectively determined then it has solvable membership problem. 
The Gupta--Sidki $p$-groups, the Grigorchuk group and, more generally, groups with finite  $L$-presentations are examples of groups
all of whose finite quotients can be effectively determined (see, for instance, \cite{Hartung_Cosetenumeration}).

The proofs of our main theorems both rely on an auxiliary result,  \autoref{thm3}, on finitely generated subgroups of the Gupta--Sidki 3-group. 
It is worth mentioning that all our results hold for the general case where $p$ is any odd prime, 
except for the length reduction argument in this  \autoref{thm3}. 
For this reason, all definitions and preliminary results are stated for the general case in  \autoref{Defs} and we only focus on the case $p=3$ for the proof of \autoref{thm3} in \autoref{prelims}. 
We depart from preliminary results in  \autoref{max section}, where we discuss maximal subgroups. 
In \cite{PervovaGrig, PervovaGupta} it was shown that all maximal proper subgroups of the Grigorchuk and Gupta--Sidki groups have finite index. 
We will show in  \autoref{maxsugps} that this maximal subgroups property passes to all finitely generated subgroups of the Gupta--Sidki 3-group, as a consequence of \autoref{mainthm}.
The proofs of \autoref{mainthm} and \autoref{thm2} are presented in the final \autoref{pfthm1}. 
The arguments generalize easily to the case $p>3$ assuming that the analogue of \autoref{thm3} holds.
In this final section we also establish the following characterization of finite subgroups of the 3-group, 
using the analysis carried out in  \autoref{prelims}:

\begin{thm}\label{finitesubgps}
Let $H$ be a finitely generated subgroup of $G$. 
Then $H$ is finite if and only if no vertex section of $H$ is equal to $G$. 
\end{thm}	
Minor changes in the proof of this theorem and the detailed analysis carried out in \cite{griwil} yield 
an identical characterization of the finite subgroups of the Grigorchuk group.
\begin{thm}\label{grigfinite}
Let $H$ be a finitely generated subgroup of the Grigorchuk group $\Gamma$. 
Then $H$ is finite if and only if  no vertex section of $H$  is equal to $\Gamma$. 
\end{thm}

\section{Definitions and preliminaries}\label{Defs}

We begin by defining the trees on which our groups act, the automorphism groups of these trees, and some of their subgroups which will be used in our proofs. 

For an integer $d\geq 2$, we may think of the vertices of the \emph{$d$-regular rooted tree $T$} as finite words over the alphabet $\{0,\ldots, d-1\}$.
We think of the empty word $\varepsilon$ as the root. The words $u, v$ are joined by an edge if $v=uw$ (or $u=vw$) for some $w$ in the alphabet. 

We can impose a metric on $T$ by assigning unit length to each edge.
Then vertex $v$ will be at distance $n$ from vertex $u$ if the unique path joining them consists of $n$ edges.
The distance of a vertex $v$ from the root is the \emph{level of $v$}.
The set of vertices of level $n$ is called the \emph{$n$th layer} of $T$ and is denoted by $\mathcal{L}_n$.

For a vertex $v\in \mathcal{L}_n$, the subtree consisting of vertices of level $m\geq n$ separated from the root by $v$ is the \emph{subtree rooted at $v$} and it is denoted by $T_v$.
Since $T$ is regular, for every vertex $v$ there is an obvious isomorphism from $T$ to $T_v$ taking $u$ to $vu$.
We will identify all $T_v$ with $T$ in the rest of the paper.

An \emph{automorphism} of a rooted tree $T$ is a permutation of the vertices that preserves the adjacency relation.
We denote the group of all automorphisms of $T$ by $\Aut T$. 
For any vertex $v$ of $T$ write $v^g$ for the image of $v$ under $g\in \Aut T$. 
We write
$$\St(v):=\{g\in \Aut T \mid v^g=v\}$$
for the \emph{stabilizer} of $v$.
The subgroup $$\St(n):=\bigcap_{v\in\mathcal{L}_n} \St(v)$$
 is the $n$th \emph{level stabilizer}.
For any subgroup $\Gamma\leq \Aut T$,  denote by $\St_{\Gamma}(v)$ and $\St_{\Gamma}(n)$  
the intersection of $\Gamma$ with the above subgroups. 
The level stabilizers $\St_{\Gamma}(n)$ have finite index in $\Gamma$.
We say that $\Gamma$ has the \emph{congruence subgroup property} if every finite index subgroup of $\Gamma$ contains some $\St_{\Gamma}(n)$.

For every $x\in \St(v)$ there is a unique automorphism $x_v\in \Aut T$ which is simply the restriction of $x$ to the subtree $T_v(=T)$.
Hence for every subgroup $\Gamma\leq \Aut T$ and every vertex $v$ of $T$ this restriction yields a homomorphism
$$\varphi_v:\St_{\Gamma}(v)\rightarrow \Aut T, \quad x\mapsto x_v.$$
The image of this homomorphism is denoted by $\Gamma_v$ and called the \emph{vertex section} of $\Gamma$ at $v$
(some authors call it an upper companion group).
We may also refer to  $x_v=\varphi_v(x)$ as the vertex section of $x$ at $v$.
Observe that if $v=uw$ then $\varphi_v=\varphi_u \circ \varphi_w$.

Notice that, although the image $\varphi_v(\St_{\Gamma}(v))$ is a subgroup of $\Aut T$, it may not be a subgroup of $G$.
We say that a subgroup $\Gamma\leq \Aut T$ is \emph{fractal} if $\varphi_v(\St_{\Gamma}(v))=\Gamma$ for every vertex $v$ of $T$.

For every $n$, define 
\begin{align*}
\psi_n:\St_{\Gamma}(n) &\rightarrow \prod_{v\in\mathcal{L}_n} \psi_v(\St_{\Gamma}(v)) \leq \prod_{v\in\mathcal{L}_n} \Aut T_n\\
 x&\mapsto (x_v)_{v\in\mathcal{L}_n}.
\end{align*}
This is an embedding, so we may  identify elements of $\St_{\Gamma}(n)$ with their image under $\psi_n$.
For the first level we usually omit the subscript and just write $\psi$.
In fact, we often (following the custom in the literature) omit the $\psi$ altogether for elements of the first level stabilizer.

We can now introduce the family of Gupta--Sidki $p$-groups.
Let $p>2$ be a prime and let $T$ be the $p$-regular rooted tree.
The \emph{Gupta--Sidki $p$-group} $G=\langle a , b \rangle$ is the subgroup of $\Aut T$ generated by two automorphisms $a$ and $b$.
The automorphism $a$ cyclically permutes the first layer as $(1\; 2\; \ldots\; p)$.
The element $b=(a,a^{-1},1,\ldots,1,b)$ is recursively defined by its action on subtrees rooted at the first layer: it acts on $T_0$ as $a$ acts on $T$, as $a^{-1}$ on $T_1$, as $b$ on $T_{p-1}$ and trivially on the other subtrees. 
The figure below shows the action of  $b$ on $T$ for $p=3$:

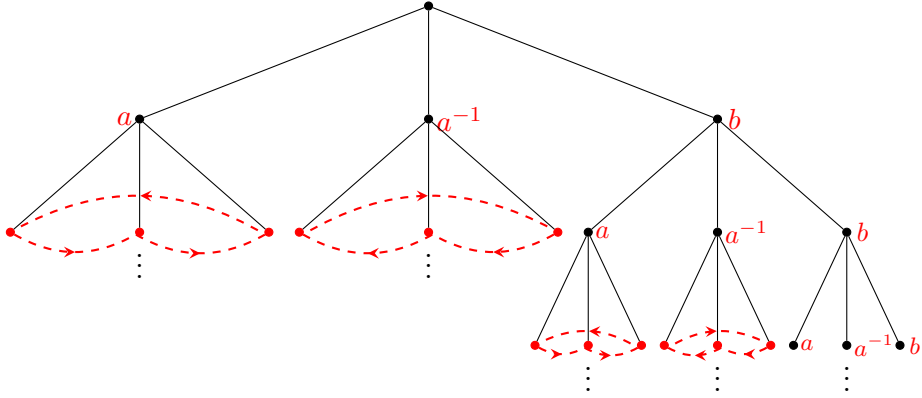
\begin{center}
\tikzset{
  on each segment/.style={
    decorate,
    decoration={
      show path construction,
      moveto code={},
      lineto code={
        \path [#1]
        (\tikzinputsegmentfirst) -- (\tikzinputsegmentlast);
      },
      curveto code={
        \path [#1] (\tikzinputsegmentfirst)
        .. controls
        (\tikzinputsegmentsupporta) and (\tikzinputsegmentsupportb)
        ..
        (\tikzinputsegmentlast);
      },
      closepath code={
        \path [#1]
        (\tikzinputsegmentfirst) -- (\tikzinputsegmentlast);
      },
    },
  },
  mid arrow/.style={postaction={decorate,decoration={
        markings,
        mark=at position .5 with {\arrow[#1]{stealth}}
      }}},
}

\begin{tikzpicture}
	\tikzstyle{every node}=[fill=black,circle,inner sep=0pt, minimum size=3pt]

	\tikzstyle{level 1}=[sibling distance=38mm]
	\tikzstyle{level 2}=[sibling distance=17mm]
	\tikzstyle{level 3}=[sibling distance=7mm]

	\node[circle, draw] (root) {}
	child {
		node[circle,draw, label=left:\textcolor{red}{$a$}] (level0start) {}
		child { node[circle, draw, red] (11) {} }
		child { node[circle, draw, red, label=below:$\vdots$] (12) {} }
		child { node[circle, draw, red] (13) {} }
	}
	child {
		node[circle,draw, label=right:\textcolor{red}{$a^{-1}$}] (level0middle) {}
		child { node[circle, draw, red] (21) {} }
		child { node[circle, draw, red, label=below:$\vdots$] (22) {} }
		child { node[circle, draw, red] (23) {} }
	}
	child {
		node[circle,draw, label=right:\textcolor{red}{$b$}] (level0end) {}
		child { node[circle, draw, label=right:\textcolor{red}{{\small$a$}}] {} 
			child { node[circle, draw, red] (311) {} }
			child { node[circle, draw, red, label=below:$\vdots$] (312) {} }
			child { node[circle, draw, red] (313) {} } 
			}
		child { node[circle, draw, label=right:\textcolor{red}{{\small$a^{-1}$}}] {} 
			child { node[circle, draw, red] (321) {} }
			child { node[circle, draw, red, label=below:$\vdots$] (322) {} }
			child { node[circle, draw, red] (323) {} }
			}
	child { node[circle, draw, label=right:\textcolor{red}{{\small$b$}}] {} 
			child { node[circle, draw, label=right:\textcolor{red}{{\footnotesize$a$}}] (331) {} }
			child { node[circle, draw, label=right:\textcolor{red}{{\footnotesize$a^{-1}$}},label=below:$\vdots$] (332) {} }
			child { node[circle, draw, label=right:\textcolor{red}{{\footnotesize$b$}}] (333) {} } 
			}
	};

	\path[draw=red,thick,dashed, postaction={on each segment={mid arrow=red}}]
   
(11) to[bend right] (12) to[bend right] (13) to[bend right] (11)
(23) to[bend left] (22) to[bend left] (21) to[bend left] (23)
(311) to[bend right] (312) to[bend right] (313) to[bend right] (311)
(323) to[bend left] (322) to[bend left] (321) to[bend left] (323);

\end{tikzpicture}

\end{center}


We establish some preliminary results and properties of the Gupta--Sidki $p$-groups that will be useful later on. \\

\noindent \textbf{Notation.}
From now on, $G$ will usually denote the Gupta--Sidki $p$-group unless otherwise stated, and we may omit the subscript in $\St_G(n)$, $\St_G(v)$.
The derived subgroup of $G$ will be denoted by $G'$ and the direct product of $i$ copies of $G$ by $G^{(\times i)}$. \\

It is immediate from the definition of $G$ as a subgroup of $\Aut T$ that $G$ is residually finite
(as $\bigcap_{n=1}^{\infty}\St_G(n)=1$). 
It follows from \cite{GS_embeddingpgroups} that $G$ is \emph{just infinite}; that is, $G$ is infinite but all its proper quotients are finite.

For $i=1,\ldots,p-1$, we write 
$$b_i:=b^{a^i}=(1,\ldots,b,a,a^{-1},\ldots,1)$$ 
where $b$ is in the $i$th co-ordinate and $b_0:=b$.
Let $B$ denote $\langle b^G \rangle$, the normal closure of $b$. 
Then it is easy to see that $B=\langle b^G \rangle$ is equal to $\langle b,b_1,\ldots,b_{p-1}\rangle=\St_G(1)$.
Note also that $b$ has order $p$.

We include the proof of the following to illustrate the use of projection arguments since these will play an important role later on.

\begin{prop}[\cite{guptasidki}]\label{subdirect}
For every $n$, the $n$th level stabilizer $\St_G(n)$ is a subdirect product of $p^n$ copies of $G$.
Therefore $G$ is infinite and  fractal.
\end{prop}

\begin{proof}
We proceed by induction.
 For $n=1$ the map $\psi:\St_G(1)\rightarrow G^{(\times p)}$ is a homomorphism.
We claim that $\varphi_i(\St_G(1))=G$ for $i\in\{0,\ldots,p-1\}$. 
This is easily seen by examining the images of generators of $B$:

\begin{center}
	\begin{tabular}{c}
		$\psi(b)=(a,a^{-1},1,\ldots,b)$,\\
		$\psi(b_1)=(b,a,a^{-1},1,\ldots,1)$, \\
		\vdots \\
		$\psi(b_{p-1})=(a^{-1},1,\ldots,1,b,a)$.
	\end{tabular}
\end{center}

 Suppose that the result is true for $n\geq1$.
 Then a similar argument as for $\St_G(1)$ shows that $\St_G(n+1)$ is a subdirect product of $p$ copies of $\St_G(n)$, each of them a subdirect product of $p^n$ copies of $G$.
This immediately shows that $G$ is fractal and infinite.
\end{proof}

The following are easy generalizations to $p>3$ of  results proved in \cite{sidkisubgroups} for $p=3$.

\begin{prop}\label{facts}
We have
\begin{enumerate}[label={\rm(\roman*)}]
\item \label{abelianise} $G/G'=\langle aG', bG'\rangle \cong C_p\times C_p$;
\item\label{B'} $\psi(B')=(G')^{(\times p)}$;
\item\label{wreath} $G/B'\cong C_p\wr C_p$.
\end{enumerate}
\end{prop}

\begin{proof}
The first item is a straightforward verification.
For the second,we have 
\begin{align*}
&\psi([b_0b_1,b_1^{-1}b_{p-1}]) \\
			&=[(a,a^{-1},1,\ldots,b)(b,a,a^{-1},\ldots,1), (b^{-1},a^{-1},a,\ldots,1)(a^{-1},1,\ldots,b,a)]\\
			&=([ab,b^{-1}a^{-1}],[1,a^{-1}],[a^{-1},a],\ldots,[1,b],[b,a])\\
			&=(1,\ldots,1,[b,a])
\end{align*}
and it is easy to see that, for $i=1,\ldots,p-1$,
$$\psi([b_0b_1,b_1^{-1}b_{p-1}]^{a^i})=\psi([b_ib_{i+1},b_{i+1}^{-1}b_{p+i-1}]) =(1,\ldots,[b,a],\ldots,1)$$
where $[b,a]$ is in the $i$th coordinate.

From the above we obtain 
$$G/B'=G/(G')^{(\times p)}=\langle a\rangle B/(G')^{(\times p)}\cong C_p\wr C_p.$$
\end{proof}

We are now able to show a commensurability property, which will be used in the next sections without mention.
\begin{prop}\label{comm}
For every odd prime $p$, if $i\equiv j \mod p-1$ then $G^{(\times i)}$ is commensurable with $G^{(\times j)}$.
\end{prop}

\begin{proof}
For fixed $i\in\{0,\ldots,p-1\}$, we show by induction on $n$ that $G^{(\times i)}$ is commensurable with $G^{(\times n(p-1)+i)}$. 
For the base case, we deduce from parts \ref{B'} and \ref{wreath} of \autoref{facts} that $G^{(\times p-1+i)}=G^{(\times p)}\times G^{(\times i-1)}$ is commensurable with  $G\times G^{(\times i-1)}= G^{(\times i)}$.
Now suppose the claim holds for $n$.
Then 
$G^{(\times (n+1)(p-1)+i)} = G^{(\times n(p-1) + i)}\times G^{(\times p-1)}$ is commensurable with $G^{(\times i)}\times G^{(\times p-1)}$, which is in turn commensurable with $G^{(\times i-1)}\times G=G^{(\times i)}$.
Hence $G^{(\times (n+1)(p-1)+i)}$ is commensurable with $G^{(\times i)}$ and our claim follows by induction.
\end{proof}

The following will be very useful.

\begin{prop}\label{2.2.2}
An element $g\in G$ is in $B$ if and only if there exist  $n_0,\ldots,n_{p-1}\in \{0,\ldots,p-1\}$ and  $c_0,\ldots,c_{p-1}\in G'$ such that
$$g=(a^{-n_{p-1}+n_0}b^{n_1}c_0, a^{-n_0+n_1}b^{n_2}c_1,\ldots, a^{-n_{p-2}+n_{p-1}}b^{n_0}c_{p-1}).$$
Moreover, when this is the case, the representation is unique.
\end{prop}	

\begin{proof}
The `if' direction is obvious.
If $g\in B$ 
then $g$ is a product of $b_0,\ldots,b_{p-1}$ and
 $$B/B'=\{b_0^{r_0}b_1^{r_1}\cdots b_{p-1}^{r_{p-1}}B' \mid r_i\in \{0,\ldots,p-1\}\}.$$ 
Thus  for some  $n_i\in\{0,\ldots,p-1\},\, c\in B'$ we have 
\begin{align*}
g &=b_0^{n_0}b_1^{n_1}\cdots b_{p-1}^{n_{p-1}}c \\
	&= (a^{n_0}b^{n_1}a^{-n_{p-1}}d_0, a^{-n_0+n_1}b^{n_2}d_1,\ldots,b^{n_0}a^{-n_{p-2}+n_{p-1}}d_{p-1}) \\
	&= (a^{n_0}a^{-n_{p-1}}b^{n_1}c_0, a^{-n_0+n_1}b^{n_2}c_1,\ldots,a^{-n_{p-2}+n_{p-1}}b^{n_0}c_{p-1}) 
\end{align*}
where $c_i,d_i\in G'$ (since $B'=(G')^{(\times p)}$). 

The $n_i$ are uniquely determined and therefore so are the $c_i$. 
\end{proof}

\begin{lem}\label{G'>St2} 
The following assertions hold:
 \begin{enumerate}[label=(\roman*)]
  \item $b^{(1)}=(b,b,\ldots,b)\in G';$
  \item $G'\geq \St(2)$.
 \end{enumerate}
\end{lem}

\begin{proof}
  For the first item note that $bb_1\cdots b_{p-1}\equiv b^{(1)} \mod B'$ because
  $bb_1\cdots b_{p-1}=(aba^{-1},b,\ldots,b)$.
  Therefore it suffices to show that $bb_1\cdots b_{p-1} \in G'$.
  To see this, observe that for $i=0,\ldots, p-1$ we have $b_i=b[b,a^i]$ so that
  $$bb_1\cdots b_{p-1}= b^2[b,a]b^{-2} b^3[b,a^2]b^{-3}\cdots b^{-1}[b,a^{-2}]b[b,a^{-1}] \in G',$$
  as required.
  
  For the second item, let $g=(g_0,\ldots,g_{p-1})\in \St(2)$, so $g_i\in \St(1)$ for each $i$. 
  \autoref{2.2.2} implies  that $g=(b^nc_0,\ldots b^nc_{p-1})\equiv (b^{(1)})^n \mod B'$ for some $n$. 
  Thus $g\in G'$ by the previous part.
\end{proof}

The proofs of \autoref{mainthm} and \autoref{thm2} use the fact that the Gupta--Sidki 3-group has the congruence subgroup property. 
This is stated in \cite[Proposition 8.4]{BartCongruence},
but the proof given there is not quite clear.
Here is a proof which works for any odd prime $p$. 
\begin{prop}\label{CSP}
The Gupta--Sidki $p$-group $G$ has the congruence subgroup property. 
\end{prop}

\begin{proof}
By \cite[Proposition 3.8]{BartCongruence}, it suffices to show that $G''$ contains some level stabilizer $\St(m)$.
The case $p>3$ follows from \cite[Lemma 2]{GrigWilPride-equivalence}, 
which shows that $G''\geq G'\times\cdots\times G'$, the direct product of $p^2$ copies of $G'$. 
Thus, by \autoref{G'>St2} we have $G''\geq \St(2) \times\cdots\times \St(2) \geq \St(4)$.

For the case $p=3$ we must prove an analogous version of \cite[Lemma 2]{GrigWilPride-equivalence}.
Let $C=[G',G]$ be the third term of the lower central series of $G$.
Since $G'$ contains $B'=G'\times G'\times G'$ and the elements $[a,b]^{a^{-1}}=(a, ab, b^{-1}a)$, $b^{(1)}=(b,b,b)$, we have
$$([x,a],1,1)\in G'' \text{ and } ([x,b],1,1) \in G''$$
for any $x\in G'$. 
Therefore $G''\geq C\times 1\times 1$ as 
$C=\langle \{[x,a], [x,b] \mid x\in G'\}^G\rangle $ and $G''\lhd \St(1)$.
Conjugating by suitable powers of $a$
we obtain that $G''\geq C\times C\times C$. 
Now, $[[b^{-1},a],b_1b_2]=(1,[a,b],1)\in C$
and from this we conclude that $C\geq G'\times G'\times G'$.
Thus $G''\geq \St(4)$ as above.
\end{proof}

This property will be important for the proofs of \autoref{mainthm} and \autoref{maxsugps} but 
it is also obviously useful for the study of finite quotients of $G$.
In this direction, we point out that in  \cite{ggscongruence}
the authors give an explicit formula for the indices $|\Gamma:\St_{\Gamma}(n)|$
not just for the Gupta--Sidki $p$-group, 
but a more general class of groups $\Gamma$ which act on  $p$-regular rooted trees (\emph{GGS groups}).
They also prove, using different methods,  some of the properties stated in the previous lemmas for arbitrary GGS groups.

The following lemma will be essential in the proof of \autoref{thm3}.

\begin{lem}\label{roadmap}
Let $H$ be a subgroup of $G$ which is not contained in $\St(1)$.
 Then  either all first level vertex sections of $H$ are equal to $G$, 
 or they are all contained in $\St(1)$ so that $\St_H(1)=\St_H(2)$.
\end{lem}

\begin{proof}
Denote by $H_0,\dots, H_{p-1}$ the first level vertex sections of $H$.
We claim that they are all conjugate in $G$.
Since $H$ is not contained in $B$ there exists $as\in H$ with $s=(s_0,\ldots,s_{p-1})\in B$. 
Thus, for every $h=(h_0,\ldots,h_{p-1})\in \St_H(1)$ we have $h^{as}=(h_{p-1}^{s_0},h_0^{s_1},\ldots,h_{p-2}^{s_{p-1}})\in \St_H(1)$.
From this we see that $H_i=H_{i-1}^{s_i}$ for $i=0,\ldots,p-1$.
The claim follows repeating this argument with powers of $as$.

Hence we may assume that no $H_i$ is equal to $G$.
We examine the image of $H$ modulo $B'$ (equivalently, the images of the $H_i$ modulo $G'$). 
Since all these vertex sections are conjugate in $G$, they must have the same images modulo $G'$ 
and the possibilities for these are $\langle ab^k\rangle$ for $k\in\mathbb{F}_p$, or $\langle b\rangle$, or $\langle a\rangle\langle b\rangle$. 
If $H_iG'=\langle a\rangle\langle b\rangle G'$ then $H_i=G$. 
Suppose not, then $H_i$ is contained in some maximal subgroup $M<G$, which by \cite{PervovaGupta} has index $p$ in $G$. 
Thus $M\geq G'$ and $G=\langle a\rangle\langle b\rangle G'=H_iG'\leq MG'=M$, a contradiction. 

So either $H_iG'=\langle b\rangle G'$ for all $i$ or there exists $k\in\mathbb{F}_p$ such that $H_iG'=\langle ab^k\rangle G'$ for all $i$.
Pick some $h=(h_0,\ldots,h_{p-1})\in\St_H(1)$ and consider its image modulo $B'$. 
If there exist $i\neq j$ such that $h_iG'$ and $h_jG'$ lie in different cyclic subgroups of 
$G/G'$
then, conjugating by suitable elements of $H\setminus \St(1)$ we would obtain $H_iG'=H_jG'=\langle a\rangle\langle b\rangle G'$, a contradiction to the above. 
Thus all $h_iG'$ lie in the same cyclic subgroup of $G/G'$. 
Our ultimate goal is to show that this cyclic subgroup is $\langle b\rangle G'$, 
so suppose for a contradiction that there is some $k\in\mathbb{F}_p\setminus\{0\}$ 
and some $(r_0,\dots,r_{p-1})\in\mathbb{F}_p^p$
such that 
$$(h_0G',\dots,h_{p-1}G')=((ab^k)^{r_0}G',\dots, (ab^k)^{r_{p-1}}G').$$
Now, by \autoref{2.2.2},
$$(h_0G',\ldots, h_{p-1}G')=(a^{-n_{p-1}+n_0}b^{n_1}G', a^{-n_0+n_1}b^{n_2}G',\ldots, a^{-n_{p-2}+n_{p-1}}b^{n_0}G')$$
for some  $n_i\in\mathbb{F}_p$. 
Comparing these representations of $hB'$, we obtain the following equations describing the exponents of $a$ and $b$:
$$(n_0,\dots,n_{p-1})C=(r_0,\dots,r_{p-1}) \textrm{ and } (n_0,\dots,n_{p-1})P=k(r_0,\dots,r_{p-1})$$
where $$C=\begin{pmatrix}
		   1 & -1  & 0  & \dots  & 0  \\
		   0 & 1   & -1 & \dots  & 0 \\
		   \vdots & \vdots &\ddots &\ddots  & \vdots  \\
		   0 & 0 &\dots & 1 &-1\\
		   -1 & 0  & 0  &\dots & 1
         \end{pmatrix}$$        
and $P$ is the permutation matrix associated to $(1\, 2\, \dots\, p)$.
These are equivalent to 
$$(n_0,\dots,n_{p-1})(kCP^{-1}-I)=(0,\dots,0).$$ 
The matrix $kCP^{-1}-I$ is a circulant matrix over $\mathbb{F}_p$ 
with non-zero determinant. 
Hence $(n_0,\dots,n_{p-1})=(0,\dots,0)$ is the only solution and $h\in B'$, which finishes the proof of the lemma. 
\end{proof}

The proof of \autoref{mainthm} relies on a word length reduction argument.
 Instead of the usual word length for finitely generated groups, we use a length function which only takes into account the number of conjugates of $b$.

Let $g\in G$ and let $a^{s_0}b^{r_1}a^{s_1}\cdots b^{r_m}a^{s_m}$ be a shortest (in the usual sense) word in $\{a, b\}$ representing $g$.
We can rewrite this word as 
$$b_{i_1}^{r_1}\cdots b_{i_m}^{r_m}a^v$$
where $b_{i_j}\neq b_{i_j+1}$ for each  $j$  and $v\in\{0,\ldots, p-1\}.$ 

Define the \emph{length} $l(g)$ of $g$ to be $m$. Thus $l(g)$ is the number of conjugates of $b$ in a shortest word in $a,b$ representing $g$.

 The following easy lemma shows how the length of elements of $G$ is reduced as we project down levels of the tree. Notice that this holds for any odd prime $p$. 

\begin{lem}\label{length}
Let $g\in \St(1)$ and suppose that $l(g)=m$. Then 
\begin{enumerate}[label=\rm{(\roman*)}]
\item $l(\varphi_k(g))\leq \frac{1}{2}(m+1)$ and 
\item if $g\in \St(2)$ then $l(\varphi_j(\varphi_k(g)))\leq \frac{1}{4}(m+3)$ 
\end{enumerate}
for $k,j\in \{0,\ldots,p-1\}$.
\end{lem}

\begin{proof}
Since $g\in\St(1)$, we can write it in the form $g=b_{i_1}^{r_1}\cdots b_{i_m}^{r_m}$ for some $i_j, r_j \in \{0,\ldots,p-1\}$.
Now, the image $\varphi_k(g)$ for each $k\in\{0,\ldots,p-1\}$  is, at worst, of one of the following forms:
\begin{enumerate}
\item $a^{r_1}b^{r_2}\cdots a^{r_m}$, so $l(\varphi_k(g))\leq\frac{m-1}{2}$;
\item $a^{r_1}b^{r_2}\cdots a^{r_{m-1}}b^{r_m}$, so $l(\varphi_k(g))\leq\frac{m}{2}$;
\item $b^{r_1}a^{r_2}\cdots b^{r_{m-1}}a^{r_m}$, so $l(\varphi_k(g))\leq\frac{m}{2}$;
\item $b^{r_1}a^{r_2}\cdots b^{r_m}$, so $l(\varphi_k(g))\leq\frac{m+1}{2}$.
\end{enumerate}
This proves the first item.

It now follows that for any $g\in \St(2)$ we have 
$$l(\varphi_j(\varphi_k(g)))\leq \frac{1}{2} \left( \frac{m+1}{2}+1 \right) =\frac{1}{4}(m+3),$$
 for every $j,k\in\{0,\ldots,p-1\}$. 
\end{proof}

\section{Maximal subgroups}\label{max section}

In this section we establish some results about groups whose maximal subgroups all have finite index.
Once \autoref{mainthm} is proved, these results will show that, for each finitely generated subgroup of the Gupta--Sidki 3-group, all maximal subgroups have finite index.
The results and proofs are analogous to those in \cite{griwil}.\\

\noindent\textbf{Notation.}
We will write $H\lf \Gamma$  and $H\ls \Gamma_1 \times \cdots \times \Gamma_n$ to mean, respectively,
that $H$ is a finite index subgroup of $\Gamma$ and that  $H$ is a subdirect product of $\Gamma_1 \times \cdots \times \Gamma_n$.

\begin{lem}\label{Lemma1}
Let $\Gamma$ be an infinite finitely generated group and $H$ a subgroup of finite index in $\Gamma$. 
If $\Gamma$ has a maximal subgroup $M$ of infinite index then $H$ has a maximal subgroup of infinite index containing $H\cap M$.
\end{lem}

\begin{proof}
We first show that any proper subgroup of $H$ containing $H\cap M$ must be of infinite index in $H$.
Suppose for a contradiction that there is a proper finite index subgroup $L$ of $H$ containing $H\cap M$.
Then $K$, the normal core of $L$ in $\Gamma$, is of finite index in $\Gamma$ and therefore $KM\leq \Gamma$ must be of finite index too.
Now, $M$ is a maximal subgroup of $\Gamma$ contained in $KM$, so either $M=KM$ or $KM=\Gamma$.
 If the former holds, then $M$ is of finite index in $\Gamma$, a contradiction. 
If the latter is true, we obtain $H=K(M\cap H)\leq L$, contradicting the assumption that $L$ is a proper subgroup of $H$.

Since $H$ is finitely generated, every proper subgroup is contained in a maximal subgroup 
(this can be shown without using Zorn's Lemma, see \cite{Neumann01041937})
and, by the above, the maximal subgroup containing $H\cap M$ must be of infinite index in $H$. 
\end{proof}

Recall that a \emph{chief factor} of a group $\Gamma$ is a minimal normal subgroup of a quotient group of $\Gamma$. 

\begin{lem}[\cite{griwil}, Lemma 3]\label{chief factors}
Let $\Gamma_1,\cdots,\Gamma_n$ be groups with the properties that all chief factors are finite and all maximal subgroups have finite index.
 If $\Delta\ls\Gamma_1\times\cdots\times\Gamma_n$ then all chief factors of $\Delta$ are finite and all maximal subgroups of $\Delta$ have finite index. 
\end{lem}

\begin{thm}\label{maxsugps}
Let $\Gamma:=G^{(\times k)}$ be the direct product of $k$ copies of $G$. 
If $H$ is a group commensurable with $\Gamma$ then all maximal subgroups of $H$ have finite index in $H$.
\end{thm}

\begin{proof}

By definition of commensurability, there exist $K\lf H$ and $J\lf\Gamma$ with $K$ isomorphic to $J$.
For $i=1,\ldots,k$, let $G_i$ denote the $i$th direct factor of $\Gamma$.
Then for every $i$ the subgroup $J_i:=J\cap G_i$ has finite index in $G_i$.
As $G$ has the congruence subgroup property, for each $i$ there is some $n_i$ with $1\times\cdots\times \St(n_i)\times\cdots\times 1\leq J_i$ and so $S:=\St(n_1)\times\cdots\times\St(n_k)\lf J$. 

Now, $G$ is residually finite and just infinite and so all of its chief factors are finite. Hence, by the main result in \cite{PervovaGupta}, all of its maximal subgroups have finite index.
Furthermore, we saw in \autoref{subdirect} that $\St(n)$ is a subdirect product of $p^n$ copies of $G$. 
Thus $S$ satisfies the assumptions of \autoref{chief factors} and all of its maximal subgroups are of finite index. 
\autoref{Lemma1} then implies that all maximal subgroups of $J$ must have finite index in $J$. 
The result now follows on applying \autoref{Lemma1} to $K\cong J$ and $H$. 
\end{proof}

\section{The case $p=3$: a key theorem}\label{prelims}

\noindent\textbf{Notation.}
From now on  we restrict our attention to the case $p=3$, so $G$ will denote the Gupta--Sidki 3-group.
Recall that $H_v$ denotes the vertex section of a subgroup $H\leq \Aut T$ at vertex $v$; 
that is, $H_v=\varphi_v(\St_H(v))=\varphi_{u_n}\circ\cdots\circ\varphi_{u_1}(\St_H(v))$ where $v=u_1\dots u_n$
is considered as a string of letters $u_i\in\{0,\ldots,p-1\}$.

\begin{thm}\label{thm3}
Let $\mathcal{X}$ be a family of subgroups of $G$ satisfying
\begin{enumerate}[label={\rm(\Roman*)},ref=(\Roman*)]
\item\label{prop1} $1\in \mathcal{X}$, $G\in \mathcal{X}$;
\item\label{prop2} if $H\in \mathcal{X}$ then $L\in \mathcal{X}$ for all $L\leq G$ such that $H\lf L$;
\item\label{prop3} if $H$ is a finitely generated subgroup of $\St(1)$ and
all first level vertex sections of $H$ are in $\mathcal{X}$ then $H\in \mathcal{X}$.
\end{enumerate}
Then all finitely generated subgroups of $G$ are in $\mathcal{X}$.
\end{thm}

\begin{proof}
Note that if $\mathcal{X}$ satisfies properties \ref{prop1}---\ref{prop3} then so does the subfamily
$\{H \mid  H^g \in \mathcal{X} \text{ for all } g \in G\}$.
We may therefore replace $\mathcal{X}$ by this subfamily and assume that if $H\in \mathcal{X}$ so is every $G$-conjugate of $H$.

Suppose for a contradiction that there are finitely generated subgroups of $G$ which are not in $\mathcal{X}$.
 Choose among them some subgroup $H$ generated by a finite set $S$ such that $D=\max \{l(s) \mid s \in S\}$ is as small as possible.

If $H\leq \St(1)$ then by \ref{prop3} at least one of the first level vertex sections of $H$ is not in $\mathcal{X}$ and
the generating set of this vertex section has elements of length at most $\frac{1}{2}(D+1)<D$ by \autoref{length}, 
contradicting the choice of $H$. 

Therefore $H$ is not contained in $\St(1)$.
We will show that there exists some $v\in\mathcal{L}_2$ such that 
the vertex section $H_v$ is not in $\mathcal{X}$ and has a generating set consisting of elements of length less than $D$.

Since $\St_H(1)$ is a finitely generated subgroup of $\St(1)$,
\ref{prop3}  implies that not all first level vertex sections of $\St_H(1)$ are in $\mathcal{X}$.
 However, if one of them is in $\mathcal{X}$ then, as all first level vertex sections of $H$ are conjugate in $G$, 
 they must all be in $\mathcal{X}$. 
Thus no first level vertex section of $H$ is in $\mathcal{X}$; in particular, none of them is equal to  $G$,
and \autoref{roadmap} asserts that they are all contained in $\St(1)$.
For each $k\in\{0,1,2\}$, property \ref{prop3} again implies that one of $H_{k0},H_{k1},H_{k2}$ is not in $\mathcal{X}$.
We claim that for some $k$ every such vertex section is generated by elements of length less than $D$.

Pick some element $t\in S\setminus \St(1)$.
 Then, as $\St_H(1)$ has index 3 in $H$, the set $T:=\{1,t,t^{-1}\}$ is a Schreier transversal to $\St_H(1)$ in $H$. 
Consequently,  $\St_H(1)=\St_H(2)$ is generated by 
$$X=\{t_1st_2^{-1} \mid t_i\in T, s\in S, t_1st_2^{-1}\in \St_H(1)\}.$$
The elements of this set have length at most $3D$, 
so the second level vertex sections of $H$ are generated by elements of length at most $(3D+3)/4$, by \autoref{length}.
Our claim follows if $D>3$. 
For $D=3$ and $D=2$ the claim follows from \autoref{D=3} and \autoref{D=2}, respectively, 
while \autoref{D=1} shows that if $D=1$ then $H\in \mathcal{X}$.
\end{proof}

\begin{lem}\label{D=1}
 Let $H$ be a subgroup of $G$ with a finite generating set $S$ consisting of elements of length at most 1.
 Then $H\in \mathcal{X}$.
\end{lem}
\begin{proof}
Any $s\in S$ must be of the form $a^k, \, b_i^r, \, b_i^ra^k$ where $k,r\in\{1,2\}$ and $i\in\{0,1,2\}$.
If $S$ consists of only one element then $H$ is finite and therefore in $\mathcal{X}$ by \ref{prop1} and \ref{prop2}.
Thus $S$ must contain at least two elements and
they clearly cannot all be of the form $a^k$.

If $S$ contains an element of the form $a^k$ and an element of any other form then $H=G\in\mathcal{X}$.
Suppose that all elements of $S$ are of the form $b_i^r$.
Then either $H=\langle b_0,b_1,b_2 \rangle =\St(1)$, 
so $H\in \mathcal{X}$ by \ref{prop1}, \ref{prop3} and \autoref{subdirect};
or $H=\langle b_{i}^{r}, b_{j}^{q}\rangle < \St(1)$ for $i\neq j\in\{0,1,2\}$ and $r,q\in\{1,2\}$,
so two of the first level vertex sections of $H$ are $G$ and the other one is $\langle a\rangle$;
therefore $H\in \mathcal{X}$ by \ref{prop3}.

Suppose that $S$ contains an element of the form $b_i^ra$
(we may assume that the power of $a$ is 1 as $(b_i^ra)^{-1}=b_{i+1}^{-r}a^2$).
If there is some $b_j^q \in S$ such that $j=i$ then
$b_j^{p-r}b_i^ra=a\in H$ so $H=G\in\mathcal{X}$.
If there is some $b_j^q\in S$ with $j\neq i$, then 
$b_j^q=(a^q,a^{-q},b^q)^{a^j}$ has $a^{\pm q}$ in the $i$th coordinate;
but 
$(b_i^ra)^3=b_i^rb_{i-1}^rb_{i+1}^r=(b^r,b_1^r,b^r)^{a^i}$
has $b^r$  in the $i$th coordinate,
so the vertex section $H_i$ of $H$ at vertex $i$ is $G\in \mathcal{X}$.
Since $H\nleq \St(1)$, the first level vertex sections of $H$ are conjugate in $G$, 
hence all first level vertex sections of $H$ are in $\mathcal{X}$ and $H\in \mathcal{X}$.

Finally, suppose that all elements of $S$ are of the form $b_i^ra$,
so $S$ contains elements $b_i^ra$, $b_j^qa$ and $b_i^ra(b_j^qa)^{-1}\in H$.
If $i=j$ then $a\in H$ and $H=G\in\mathcal{X}$.
If $i\neq j$ then 
$$b_i^ra(b_j^qa)^{-1}=(a^r,a^{-1},b^r)^{a^i}(a^{-q},a^q,b^{-q})^{a^j}$$ 
has $b^ra^{\pm q}$ in the $i$th coordinate
while $(b_i^ra)^3\in H$ has $b^r$ in the $i$th coordinate.
Thus $H_i=G$ and $H\in \mathcal{X}$ by the argument in the previous paragraph.
\end{proof}

\begin{lem}\label{D=2}
Let $H \nleq \St(1)$ be a subgroup of $G$ with a finite generating set $S$ consisting of elements of length at most $2$
and such that $H_u\leq\St(1)$ for all $u\in \mathcal{L}_1$.
 Then $H_v$ is generated by elements of length at most $1$ for all $v\in \mathcal{L}_2$.
\end{lem}

\begin{proof}
First note that no generator of $H$ is in $\St(1)$ 
since the only possibilities are elements of the form $b_i^r$ and $b_{i_1}^{r_1}b_{i_2}^{r_2}$ 
which are not in $\St(2)$.

If some $t\in S \setminus \St(1)$ has length less than 2 then, as in the proof of \autoref{thm3}, 
the set $T=\{1,t,t^{-1}\}$ is a transversal of $\St_H(1)$ in $H$ and $\St_H(1)$ is generated by the set 
$X=\{t_1st_2^{-1} \mid t_i\in T, s\in S, t_1st_2^{-1}\in \St_H(1)\}$.
The elements of $X$ have length at most 4, 
so by \autoref{length} all second level vertex sections of $H$ will be generated by elements of length at most 1. 

Suppose then that all elements of $S$ have length 2, that is, they are all of the form $b_{i_1}^{r_1}b_{i_2}^{r_2}a$.
It suffices to consider only elements of this form as 
$(b_{i_1}^{r_1}b_{i_2}^{r_2}a)^{-1}=b_{i_2+1}^{-r_2}b_{i_1+1}^{-r_1}a^2$.
Pick some $t=b_{i_1}^{r_1}b_{i_2}^{r_2}a\in S$ to  form the transversal $T$ so that $\St_H(1)$ is generated by $X$ as above. 
Then every element of $X$ is of the form $t^{-1}s, \, st^{-1}, \, t^3$ or $tst$ 
where $s=b_{j_1}^{q_1}b_{j_2}^{q_2}a\in S$.
We show that all possible combinations of $i_1, i_2, j_1, j_2, r_1, r_2, q_1, q_2$ 
give rise to generators of second level vertex sections of $H$ of length at most 1. 

The forms $t^{-1}s=b_{i_2+1}^{-r_2}b_{i_1+1}^{-r_1}b_{j_1+1}^{q_1}b_{j_2+1}^{q_2}$
and $st^{-1}=b_{j_1}^{q_1}b_{j_2}^{q_2}b_{i_2}^{-r_2}b_{i_1}^{-r_1}$
yield elements of length at most 4. 
Thus, by \autoref{length}, their second level vertex sections have length at most 1.

Since $i_1\neq i_2$, an element of the form 
$t^3=b_{i_1}^{r_1}b_{i_2}^{r_2}b_{i_1-1}^{r_1}b_{i_2-1}^{r_2}b_{i_1+1}^{r_1}b_{i_2+1}^{r_2}$
will have at most two separate instances of each of $b_0, b_1, b_2$. 
Hence its first level vertex sections will have length at most 2 and so the second level vertex sections have length at most 1. 

More care is required for the form 
$tst=b_{i_1}^{r_1}b_{i_2}^{r_2}b_{j_1-1}^{q_1}b_{j_2-1}^{q_2}b_{i_1+1}^{r_1}b_{i_2+1}^{r_2}$. 
Easy combinatorial arguments (using that $i_1\neq i_2$ and $j_1\neq j_2$) show that
the first level vertex sections of an element of this form cannot have 
length greater than 3 and that
the only way they can have length 3 is if $i_2=i_1-1$ and either $j_1=i_1+1$ or $j_2=i_1+1$. 
For ease of exposition, assume without loss of generality that $i_1=0$. 
Then 
\begin{align*} 
tst &=b_{0}^{r_1}b_{2}^{r_2}b_{j_1-1}^{q_1}b_{j_2-1}^{q_2}b_{1}^{r_1}b_{0}^{r_2}\\
	&=\begin{cases}
b_0^{r_1}b_2^{r_2}b_0^{q_1}b_{j_2-1}^{q_2}b_1^{r_1}b_0^{r_2}, & j_1-1=i_1 \\
b_0^{r_1}b_2^{r_2}b_{j_1-1}^{q_1}b_0^{q_2}b_1^{r_1}b_0^{r_2}, & j_2-1=i_1.
		\end{cases}
\end{align*}	
The vertex sections at vertices 0 and 1 have length at most 2, while the one at vertex 2 looks like
\begin{align*}
&\begin{cases}
b^{r_1}a^{r_2}b^{q_1}a^{ q_2}a^{-r_1}b^{r_2}, & j_1-1=i_1=j_2; \\
b^{r_1}a^{r_2}b^{q_1}a^{-q_2}a^{-r_1}b^{r_2}, & j_1-1=i_1, \, j_2=i_1-1;\\
b^{r_1}a^{r_2}a^{q_1}b^{q_2}a^{-r_1}b^{r_2}, & j_2-1=i_1=j_1; \\
b^{r_1}a^{r_2}a^{-q_1}b^{q_2}a^{-r_1}b^{r_2}, & j_2-1=i_1, j_1=i_1-1.
\end{cases} 
\end{align*}

\noindent In the first and third cases we have 
$$\varphi_2(t^{-1}s)=b^{-r_2}a^{r_1}a^{\pm q_1}a^{\mp q_2}=b^{-r_2}$$ 
(as $H_2\leq\St(1)$ by assumption);
while, in the second and fourth cases we have 
$$\varphi_2(st^{-1})=a^{\mp q_1}a^{\pm q_2}a^{-r_2}b^{-r_1}=b^{-r_1}.$$
Thus, in all cases the vertex section of $H$ at vertex 22 is $G$, which is indeed generated by elements of length at most 1. 
\end{proof}

\begin{lem}\label{D=3}
 Let $H\neq \St(1)$ be a subgroup of $G$ with a finite generating set $S$ consisting of elements of length at most 3
 and such that $H_u\leq \St(1)$ for all $u\in \mathcal{L}_1$.
 Then there exists $k\in\mathcal{L}_1$ such that $H_{kj}$ is generated by elements of length less than 3 for all $j\in \{0,1,2\}$.
\end{lem}
\begin{proof}
If there exists $t\in S\setminus\St(1)$ with $l(t)\leq 2$ then, by the same argument as in the proof of \autoref{D=2},
the generating set $X$ of $\St_H(1)$ consists of elements of length at most 7, 
the second level vertex sections of which have length strictly less than 3. 

If all generators in $S\setminus \St(1)$ have length 3,
 pick one of them, $t$, to form the transversal $T$.
Assume that $t$ has the form $b_{i_1}^{r_1}b_{i_2}^{r_2}b_{i_3}^{r_3}a$ 
 for $i_1,i_2,i_3\in\{0,1,2\}$, $i_2\neq i_1, i_2\neq i_3$ and $r_1,r_2,r_3=\pm 1$.
It suffices to consider elements of this form since $t^{-1}=b_{i_3}^{r_3}b_{i_2}^{-r_2}b_{i_1}^{-r_1}a^{-1}$.
 We may pick $k\in \{0,1,2\}$ such that neither $\varphi_k(b_{i_3+1})$ nor $\varphi_k(b_{i_1})$ is a power of $b$. 
Then the elements of $X$ have length no more than 9 and  
the choice of $k$ ensures that  $\varphi_k(x)=a^{\pm 1}\varphi_k(bz)$ for each $x\in X$ 
where $l(bz)\leq 8$.
Hence the vertex sections $H_{kj}$ for $j\in\{0,1,2\}$ are generated by elements of length at most $(8+3)/4 < 3$, as required.
\end{proof}

\section{The case $p=3$: proof of main theorems}\label{pfthm1}
Using the length reduction arguments and results in the previous section,
we easily obtain a characterization of the finite subgroups of $G$.

\begingroup
\def\thethm{\ref*{finitesubgps}}
\begin{thm}
Let $H$ be a finitely generated subgroup of $G$. 
Then $H$ is finite if and only if no vertex section of $H$ is equal to $G$. 
\end{thm}
\addtocounter{thm}{-1}
\endgroup

\begin{proof}
For the non-trivial implication, we make the crucial observation 
that for each $v\in T$, every vertex section of $H_v$ is a vertex section of $H$.
Let $H$ be generated by a finite set $S$.
We proceed by induction on $D$, the maximum length of elements in $S$. 
If $D=1$, then  by the proof of \autoref{D=1}, either $H$ is finite or $H_u=G$ for every $u\in \mathcal{L}_1$.

Assume that the assertion of the theorem holds whenever $D\leq n$ with $n\geq 1$. 
For $D=n+1$ we consider the cases $H\leq \St(1)$ and $H\nleq \St(1)$ separately.
If $H\leq \St(1)$ then each first level vertex section $H_u$ of $H$ is generated by elements of length at most $(D+1)/2=(n+2)/2<D$,
so that $H_u$ is finite by inductive hypothesis.
Thus $H$ itself must be finite as the map $\psi: H\rightarrow H_0\times H_1\times H_2$ is an injective homomorphism. 

If $H\nleq \St(1)$ then $H_u\leq \St(1)$ for every $u\in \mathcal{L}_1$ by \autoref{roadmap}.
In the case $n=1$, \autoref{D=2} shows that each second level vertex section $H_v$ is generated by elements of length at most 1,
and is therefore finite. 
Thus $H\hookrightarrow \prod_{v\in \mathcal{L}_2} H_v$ is finite.
If $n=2$, by \autoref{D=3} , there exists $u\in \mathcal{L}_1$ such that $H_{ui}$ is generated by elements of length at most 2 for all $i\in \{0,1,2,\}$.
Hence $H_k$ is finite. 
Since $H\nleq \St(1)$, all first level vertex sections of $H$ are conjugate in $G$ and are therefore finite, making $H$ finite.
Lastly, if $n>2$, by the same argument as in the proof of \autoref{thm3}, 
$H_v$ is generated by elements of length at most $3D+3/4<D$ whenever $v\in \mathcal{L}_2$.
Hence $H$ is finite and the theorem follows by induction.
\end{proof} 

The same methods as in the above proof and the analysis carried out in \cite[Thorem 3]{griwil} yield
an identical characterization of finite subgroups of the Grigorchuk group $\Gamma=\langle a,b,c,d\rangle$.
For the reader's convenience, we sketch a proof of the non-trivial implication.
The length function here is as in \cite{griwil}, namely, the usual word length for finitely generated groups.

\begingroup
\def\thethm{\ref*{grigfinite}}
\begin{thm}
Let $H\leq \Gamma$ be finitely generated. 
Then $H$ is finite if and only if no vertex section of $H$ is equal to $\Gamma$. 
\end{thm}
\addtocounter{thm}{-1}
\endgroup

\begin{proof}
Let $H$ be generated by a finite set $S$ such that $1\in S$ and $S^{-1}=S$ and let $D$ be the maximum length of elements of $S$.
We induct on $D$. 

If $D=1$ then either $H$ is finite or $H=\Gamma$, a contradiction. 

Assume that the theorem holds whenever $D\leq n$ with $n\geq 1$. 
For $D=n+1$ we consider the cases $H\leq \St(1)$ and $H\nleq \St(1)$ separately.
If $H\leq \St_{\Gamma}(1)$ then each first level vertex section $H_u$ of $H$ is generated by elements of length at most $(D+1)/2=(n+2)/2<D$,
by \cite[Lemma 7]{griwil},
so that $H_u$ is finite by inductive hypothesis.
Thus $H$ itself must be finite as the map $\psi: H\rightarrow H_0\times H_1$ is injective. 

If $H\nleq \St_{\Gamma}(1)$ there are three cases to consider depending on the vertex section $H_0$ at vertex 0. 
If $H_0 \leq \St_{\Gamma}(1)$ then, by Case 3 in the proof of \cite[Theorem 3]{griwil},
 $H_{00}$ and $H_{01}$ are generated by elements of length less than $D$.
 Thus $H_{00}$ and $H_{01}$ are finite by inductive hypothesis, so $H_0$ is finite.
Now, $H_0, H_1$ are conjugate in $\Gamma$ as $H\nleq \St_{\Gamma}(1)$; hence $H$ is finite.

If $H_0\Gamma'=\langle ad\rangle \Gamma'$ or $\langle a, d\rangle \Gamma'$, then $H_{00}\leq \St_{\Gamma}(1)$ by \cite[Lemma 6]{griwil}.
Case 2 of the proof of \cite[Theorem 3]{griwil} shows that 
$H_{000}$ and $H_{001}$ are generated by elements of length less than $D$ so that $H_{00}$ is finite.
As $H$ acts transitively on the second layer of the tree,
there exists $h=(a(s_0,s_1),h_1)\in H$ with $s_0,s_1\in \Gamma$ swapping the vertices 00 and 01.
For any $g=((g_{00},g_{01}),g_1)\in \St_H(00)$ we have 
$g^h=((g_{01}^{s_1}, g_{00}^{s_0}), g_1^{h_1})$, 
whence $H_{00}^{s_1}=H_{01}$. 
Hence $H_{01}$ is finite, making $H_0$ and therefore $H$ finite.

If $H_0\Gamma'=\langle ac\rangle \Gamma'$ or $\langle a,c\rangle \Gamma'$, then $H_{000}\leq \St_{\Gamma}(1)$ by \cite[Lemma 6]{griwil}.
Case 1 of the proof of \cite[Theorem 3]{griwil} shows that 
$H_{0000}$ and  $H_{0001}$ are generated by elements of length less than $D$ so that $H_{000}$ is finite.
Since $H$ acts transitively on the third layer, there is an element $h\in H$ swapping the vertices 000 and 001. 
Then $H_{000}$ and $H_{001}$ are conjugate in $\Gamma$, by an argument similar to the one above.
Thus $H$ is finite by the arguments in the previous case and the theorem follows by induction. 
\end{proof}

We now move on to the proof of the two main results. 
These can be easily generalized to the case $p>3$, provided that \autoref{thm3} holds.

\noindent\textbf{Notation.} For simplicity, $\mathcal{G}$ will denote `$G$ or $G\times G$'.

\begingroup
\def\thethm{\ref*{mainthm}}
\begin{thm}
Every infinite finitely generated subgroup of the Gupta--Sidki $3$-group $G$ is commensurable with $G$ or the direct square $G\times G$.
\end{thm}
\addtocounter{thm}{-1}
\endgroup

To prove this theorem it suffices to show that properties \ref{prop1}---\ref{prop3} in \autoref{thm3} hold for the class $\mathcal{C}$ of subgroups of $G$ which are finite or commensurable with $\mathcal{G}$.
 Clearly, \ref{prop1} holds as the trivial subgroup and $G$ are in this class.
 It is also easy to see that \ref{prop2} is satisfied by $\mathcal{C}$: if $H\in \mathcal{C}$ is commensurable with $\mathcal{G}$ and 
 $J\leq _fH$ is a subgroup isomorphic to a finite index subgroup of $\mathcal{G}$ then, 
 for any $L\leq G$ containing $H$ as a finite index subgroup, $J$ is also contained in $L$ with finite index. 
 Thus $L\in \mathcal{C}$. 
 If $H\in\mathcal{C}$ is finite then any $L$ containing $H$ with finite index must also be finite, so $L\in \mathcal{C}$ too.
Thus it only remains to show that $\mathcal{C}$ satisfies 
\begin{enumerate}[label=(\Roman*), start=3]
\item If $H$ is a finitely generated subgroup of $\St(1)$ and all first level vertex sections of $H$ are in $\mathcal{C}$ then $H\in \mathcal{C}$.
\end{enumerate}	
This will follow from the next lemma, which uses similar ideas to those in \cite{griwil}.

\begin{lem}
If $H\ls H_1\times \dots \times H_n$ where each direct factor $H_i$ is in $\mathcal{C}$  then $H\in \mathcal{C}$. 
In other words, $\mathcal{C}$ is closed for subdirect products.
\end{lem}

\begin{proof}
This reduces to proving that if $H/N_1, H/N_2\in \mathcal{C}$ then $H/(N_1\cap N_2)\in \mathcal{C}$; that is, the case $n=2$. 
Suppose then that $H\ls H_1\times H_2$ with $H_i\in \mathcal{C}$. 
If both factors $H_i$ are finite then so is $H$ and we are done. 
Assume that $H_1$ is commensurable with $\mathcal{G}$. 
Then there exists $K_1\lf H_1$ isomorphic to some $L_1\lf \mathcal{G}$. 
By \autoref{CSP}, $K_1$ contains some level stabilizer $\St_G(n)$
(or a direct product $\St_G(n)\times \St_G(m)$ if $K_1\lf G\times G$), 
and this is subdirect in some direct power of $G$ by \autoref{subdirect}. 
Denote by $M$ the preimage in $H$ of $\St_G{(n)}$ (or $\St_G(n)\times \St_G(m)$). 
Since $M$ has finite index in $H$, so does its projection to $H_2$. 
Thus we may replace $H$ by $M$, $H_1$ by $\St_G(n)$ (or $\St_G(n)\times \St_G(m)$) and $H_2$ by the projection of $M$ in $H_2$ and obtain
$$H\ls G_1\times \dots \times G_r\times H_2$$
for some finite $r$, which we take to be minimal, where each $G_i$ is a copy of $G$.

Write $A:=G_1\times \dots \times G_r$. 
Then $AH=G_1\times \dots \times G_r\times H_2$ is commensurable with $\mathcal{G}$ and we claim that $H$ has finite index in $AH$. 
Denote by $K_i$ the kernel of the map from $H$ to all factors of $A$ except $G_i$; that is, $K_i=H\cap (1\times \dots\times G_i\times\dots\times 1)\triangleleft H$. 
Then $K_i\triangleleft G_i$ and, since $G_i$ is just infinite, $K_i$ has finite index in $G_i$. 
This way we obtain a finite index normal subgroup $K_1\times \dots \times K_r$ of $A$ contained in $H$. 
Hence $|AH : H| = |A : A\cap H|$ is finite and our claim is proved, showing that $H$ is commensurable with $\mathcal{G}$. 

For the case $n\geq 2$, proceed as follows:
Take $H_1\cong H/\ker (H\twoheadrightarrow H_1)$ and $H_2\cong H/\ker(H\twoheadrightarrow H_2) \in \mathcal{C}$. 
By the above, 
$$H/(\ker (H\twoheadrightarrow H_1)\cap \ker (H\twoheadrightarrow H_2))\cong \mathrm{im}(H\rightarrow H_1\times H_2) \in \mathcal{C}.$$ 
Then, again by the above, 
$$H/(\ker (H\twoheadrightarrow H_1)\cap \ker (H\twoheadrightarrow H_2))\cap (\ker(H\twoheadrightarrow H_3))\cong \mathrm{im}(H\rightarrow H_1\times H_2\times H_3) \in \mathcal{C}.$$ 
At the $n$th iteration of this operation, we reach 
$$H=H/(\bigcap_{i=1}^{n-1} \ker(H\twoheadrightarrow H_i) \cap \ker(H\twoheadrightarrow H_n))\in \mathcal{C}.$$
\end{proof}

\begingroup
\def\thethm{\ref*{thm2}}
\begin{thm}
The Gupta--Sidki $3$-group $G$ is subgroup separable.
\end{thm}
\addtocounter{thm}{-1}
\endgroup

To prove this theorem it suffices to show that the conditions of \autoref{thm3} hold
for the class $\mathcal{S}$ of finitely generated subgroups of $G$ 
all of whose subgroups of finite index are closed with respect to the profinite topology on $G$. 

Clearly, $\mathcal{S}$ satisfies \ref{prop1}.
To see that it also satisfies \ref{prop2}, let $H\lf L$ for some $H$ in $\mathcal{S}$; 
thus  $L$ is finitely generated. 
For any $K\lf L$, we have $K\cap H\lf H$ so $K\cap H$ is closed in $G$ by assumption. 
But $K\cap H$ also has finite index in $K$,
hence each of its finitely many cosets is also closed in $G$ and therefore so is $K$, their union.

Before we can show that $\mathcal{S}$ also satisfies \ref{prop3} we need the following lemma.

\begin{lem}\label{lemma12}\hfill
\begin{enumerate}[label={\rm(\roman*)}]
\item Suppose that $H_0$ is a group all of whose quotients are residually finite and
that each of the groups $G_1, \ldots, G_n$ either is finite or is residually finite, just infinite and not virtually abelian.
Let $H\ls H_0\times G_1\times \cdots \times G_n$. Then every quotient of $H$ is residually finite. 

\item If $H$ is abstractly commensurable with $G$ or $G\times G$ then every quotient of $H$ is residually finite. 
\end{enumerate}
\end{lem}

\begin{proof}
This is essentially Lemma 12 in \cite{griwil}.
The first part is identical and the proof of the second only requires small modifications. 
Suppose that $K\triangleleft H$; we want to show that $K$ is an intersection of subgroups of finite index in $H$. 
Since $H$ is commensurable with $G$ or $G\times G$ and they both have the congruence subgroup property,
there is some normal subgroup $N$ of finite index in $H$ which is a subdirect product of finitely many copies of $G$. 
By the same argument as in the proof of \ref{prop2},
it suffices to show that $K\cap N$ is closed in $N$ with respect to the profinite topology on $N$. 
This follows from the first part of the lemma, as it is equivalent to $N/(K\cap N)$ being residually finite. 
\end{proof}

We may now proceed to show that $\mathcal{S}$ also satisfies \ref{prop3}. 
\begin{lem}
Let $H$ be a finitely generated subgroup of $\St(1)$ such that its first level vertex sections $H_0,H_1,H_2$ are in $\mathcal{S}$.
Then $H$ is in $\mathcal{S}$. 
\end{lem}

\begin{proof}
Let $K\lf H$, then its first level vertex sections have finite index in those of $H$ and are therefore closed in $G$. 
So we only need to show that $H$ is closed in $G$. 
We will show that $\psi(H)$ is closed in $G\times G \times G$, so
that $H$ is closed in $\St(1) \lf G$ and hence in $G$.
In fact, if suffices to show that $\psi(H)$ is closed in $H_0\times H_1\times H_2$
because in that case
$\psi(H)=\bigcap K_i$ for $K_i\lf H_0\times H_1\times H_2$ with
each $K_i$ closed in $G\times G\times G$ by our assumption on $\mathcal{S}$.
The lemma will follow from the more general

\noindent \textbf{Claim.} For every $n\in \mathbb{N}$, if $H\ls H_0\times \dots\times H_n$ with $H_i\in \mathcal{S}$ 
then $H$ is closed in $H_0\times \dots\times H_n$.\\
\noindent\emph{Proof of claim.}
We proceed by induction on $n$. 
When $n=1$, define $L_0\times 1:=H\cap(H_0\times 1)$. 
Then $L_0$ is normal in $H_0$, which is commensurable with $G$ or $G\times G$ by \autoref{mainthm}. 
By \autoref{lemma12}, $H_0/L_0$ is residually finite so
there is a collection $(N_j)_{j\in J}$ of normal subgroups of finite index in $H_0$ whose intersection is $L_0$.
Each subgroup $(N_j\times1)H$ is of finite index in $H_0\times H_1$
so it is enough to show that $H=\bigcap_{j\in J}((N_j\times1)H)$.
Let $u$ be an element of the intersection, 
say $u=(n_j,1)(h_{j,0},h_{j,1})$ for each $j\in J$. 
Then $h_{j,1}$ is constant so we can write $u=(n_jh_{j,0},h_1)$ for each $j\in J$.
Fix $i\in J$ and note that $h:=(h_{i,0},h_1)\in H$.
Then, for each $j\in J$ we have 
 $$uh^{-1}=(n_j,1)(h_{j,0}h_{i,0}^{-1},1)\in (N_j\times 1)(H\cap H_0\times 1)=N_j\times 1.$$
Thus  $uh^{-1} \in \bigcap_{j\in J} N_j\times 1 = L_0\times 1 \leq H$ and $u\in H$.

Now assume that $H$ is closed in $H_0\times \dots\times H_{n-1}$ 
whenever $H\ls H_0\times\dots\times H_{n-1}$ and each $H_i$ is in $\mathcal{S}$.
Suppose that $H\ls H_0\times\dots\times H_n$ with $H_i\in \mathcal{S}$.
Define $L_0\times 1\times\dots\times 1:=H\cap(H_0\times 1\times\dots \times 1)$. 
As in the base step, there exist normal subgroups $(N_j)_{j\in J}$  of finite index in $H_0$ whose intersection is $L_0$.
Let $H^j:=(N_j\times 1\times\dots \times 1)H$ for each $j\in J$. 
We show that $H=\bigcap_{j\in J}H^j$ and that each $H^j$ is closed in $H_0\times\dots\times H_n$.
The first statement is proved similarly to the two-factor case: 
write an element $u$ of the intersection as $u=(n_jh_{j,0},h_1,\dots,h_n)$
and find some $h:=(h_{i,0},h_1,\dots,h_n)\in H$
so that $uh^{-1}=(n_j,1,\dots,1)(h_{j,0}h_{i,0}^{-1},1,\dots,1)\in N_j\times 1\times\dots\times 1$.
For the second statement, fix $j\in J$ and let $P$ be the pre-image of $N_j$ in $H^j$ under the canonical projection to $H_0$. 
Then $P$ has finite index in $H^j$ as $N_j$ has finite index in $H_0$.
Thus, for each $i\neq 0$, the projection $P_i$ of $P$ onto the $i$th factor has finite index in $H_i$, so $P_i\in \mathcal{S}$.
Furthermore, $N_j\times1\times\dots\times1\leq P\ls N_j\times P_1\times\dots\times P_n$, 
so $$P/(N_j\times1\times\dots\times1) \ls P_1\times\dots\times P_n.$$ 
By the inductive hypothesis, $P/(N_j\times1\times\dots\times1)$ is closed in $P_1\times\dots\times P_n$.
Thus $P$ is closed in $N_j\times P_1\times\dots\times P_n\lf H_0\times H_1\times\dots\times H_n$,
hence also in $H_0\times H_1\times\dots\times H_n$. 
\end{proof}

\subsection*{Acknowledgements}
I would like to thank my supervisor, John S. Wilson, for suggesting the problem and for his continued guidance and support.

This work was carried out as part of my DPhil studies, financially supported by \emph{Fundación La Caixa (Spain)}.

\bibliographystyle{abbrv}
\bibliography{transfer}

\end{document}